\documentclass[a4paper, reqno, 11pt]{amsart}
\usepackage{preamble}
\usepackage{import}
\renewcommand{\X}{\tilde{X}}
\renewcommand{\Pproj}{\Proj}
\newcommand{\D}{\rm D}
\begin{document}
	\title{A semi-orthogonal decomposition theorem for weighted blowups}
	\author{Oliver Li}
    \address{Department of Mathematics \\
	University of Melbourne \\
	Parkville VIC 3052, Australia }
    \email{oli1@unimelb.edu.au}
    
	\begin{abstract}
		We establish a semi-orthogonal decomposition for the weighted blowup of an algebraic stack along a Koszul-regular weighted centre, generalising the classic result of Orlov. Our approach is based on the work of Bergh-Schn\"urer.
	\end{abstract}
	\maketitle
	
	\section{Introduction}
	Weighted (stacky) blowups have been getting a lot of attention in recent years, prominently as a tool for improving the efficiency of classical algorithms. For example, they have been used by Abramovich, Temkin and Włodarczyk in \cite{Abramovich_2024} to improve Hironaka's resolution of singularities in characteristic zero. Examples of weighted blowups include usual blowups and the root stack construction of Cadman \cite{root_stax}. We refer the reader to \cite{QuekRydh} for more background.
	
	A natural question to ask is how a given invariant changes after a weighted blowup. For the integral Chow ring, this has been answered by Arena, Obinna and Abramovich in \cite{chow_ring_blowups}. As the first result of this paper, we consider various (quasi-)coherent derived categories. We prove:
	\begin{theorem}\label{thm-main}
		Fix $\mathbf{d} =(d_0,...,d_n)$ a tuple of positive integers. Let $X$ be a quasi-compact algebraic stack and let $\I_\bullet$ be a Koszul-regular weighted centre on $X$, that smooth-locally is of the form $\sum_{i = 0}^n (f_i, d_i)$. Let $\X $ be the blowup of $X$ along $\cal{I}_\bullet$ and let $E$ denote the exceptional divisor. Write $Z_i = V(\cal{I}_i)$. We have a commutative diagram: \[\begin{tikzcd}
			E \arrow[r, hook,"j"]\arrow[d, "p"] & \tilde{X} \arrow[d, "\pi"] \\
			Z _1\arrow[r,hook, "i"] & X.
		\end{tikzcd} \] 
		Let $\rm{D}$ denote one of $\Dqc$ or $\Perf$ (i.e. perfect complexes). Write $\Phi_0^\rm{D}$ for the functor $L\pi^*: \rm{D}(X) \rightarrow \Dqc(\X)$ and for any $r < 0$ write $\Phi_r^\D$ for the functor \[\F \mapsto j_*(p^*(\F) \otimes \O_{E}(r)): \D(Z_1)\rightarrow \Dqc(\X). \]Then there is a semi-orthogonal decomposition with $|\mathbf{d}| := \sum d_i$ summands: \[\rm{D}(\tilde{X}) = \langle \im \Phi^\D_{1-|\mathbf{d}|},\dots,\im \Phi^\D_{-1}, \im \Phi^\D_0 \rangle. \] Moreover, if $X$ is locally noetherian, then the statement also holds for $\mathrm{D} =\DCoh^-$ and $\rm{D}=\DbCoh$.
	\end{theorem}
    
	In the two extreme cases of a weighted blowup, namely root stacks and usual blowups, Theorem \ref{thm-main} has been established, most recently by Bergh and Schn\"urer (\cite[Theorem C]{ConservDesc}). It is their approach that we will follow. However, their argument uses in two crucial ways (namely in establishing that $L\pi^*$ is fully faithful and that $\rm{D}(X)$ is generated by \enquote{just} the right number of elements) that the blowup of an algebraic stack $X$ along a Koszul-regular closed immersion $Z\hookrightarrow X$ of constant codimension (locally) embeds into $\P^{\codim(Z/X)-1}_X$. Unfortunately, it is not so clear that the analogue of this fact is true in the weighted case, so their argument is not directly applicable.
	
	In this paper, we get around the above-mentioned problem by considering a weighted blowup as the $\Pproj$ of an \textit{extended} Rees algebra, in the sense of \cite{QuekRydh}. The extended Rees algebra has a particularly attractive description (see Example \ref{eg-h1-regular-extended-rees}) and allows us to work with explicit Koszul complexes. 
    \begin{example}
        By {\cite[Example 5.4.4]{QuekRydh}}, every regular weighted blowup $\tilde{X} \rightarrow X$ factors as a sequence of root stacks, a usual (regular) blowup followed by a sequence of \enquote{de-rootings}. In other words, there is a morphism $\hat{X} \rightarrow \tilde{X}$ which is a sequence of root stacks such that the composition $\hat{X} \rightarrow X$ admits a factorisation as a (potentially different) sequence of root stacks followed by a usual (regular) blowup. Note in particular that the de-rootings are necessary, see for example the explicit example presented in \cite[Example 5.4.4]{QuekRydh}. Since the derived category of root stacks and usual blowups are classical we have a description of $\Dqc(\hat{X})$. Theorem \ref{thm-main} then yields an alternative SOD of $\Dqc(\hat{X})$. See Example \ref{eg: fkn-notation} for an explicit case worked out.
    \end{example}
    A pleasant byproduct of the work done to obtain Theorem \ref{thm-main} is that we will compute the relative dualising sheaf of a (Koszul-regular) weighted blowup. This is our second result:
    \begin{theorem}\label{thm: shriek}
        Fix $\mathbf{d} =(d_0,...,d_n)$ a tuple of positive integers. Let $X$ be a concentrated algebraic stack with quasi-finite and separated diagonal. Let $\cal{I}_\bullet$ be a Koszul-regular weighted centre locally of the form $\sum (f_i, d_i)$ and let $\pi:\tilde{X}\rightarrow X$ be the weighted blowup with exceptional divisor $E$. Then $Rf_*$ admits a right adjoint $\pi^!$ given by \[\pi^!\F = L\pi^*\F \ltimes_{\tilde{X}} \O_{\tilde{X}}((\sum d_i - 1)E). \] 
    \end{theorem}

    This, of course, generalises the well-known result for classical blowups.

	\subsection{Related Results}
	When $X$ is a smooth projective complex variety and $\tilde{X}$ is the (usual) blowup of a smooth subvariety, the result for $\mathrm{D} = \DbCoh = \Perf$ goes back to Orlov \cite[Theorem 4.1]{OrlovThesis}. This was generalised by Elagin to the case where $X$ is the quotient of a smooth quasi-projective variety by a linearly reductive group over an arbitrary field in \cite[Theorem 10.2]{Elagin_2012}. And as mentioned, when $X$ is an algebraic stack and $\tilde{X}$ is the blowup of a Koszul-regular closed substack (in the sense of \cite[\href{https://stacks.math.columbia.edu/tag/063J}{Tag 063J}]{stax}), the result was established by Bergh and Schn\"urer in \cite[Theorem C]{ConservDesc}.
	
	The case of root stacks goes back to work of Ishii and Ueda in \cite[Theorem 1.6]{ishii_ueda}, where they established the decomposition for $\rm{D} =\DbCoh = \Perf$ in the case $X$ is a smooth Deligne-Mumford stack over a field and $\tilde{X}$ is the root stack associated to a smooth divisor. This was refined to the case where $X$ was an arbitrary root stack by Bergh, Lunts and Schn\"urer as \cite[Theorem 4.7]{geometricity}, and reiterated in the aforementioned work of Bergh and Schn\"urer \cite[Theorem C]{ConservDesc}.
	
	When $X$ is a noetherian separated scheme over an algebraically closed characteristic zero field and $\X$ is a root stack, a periodicity phenomenon in the decomposition of $\DbCoh$ was observed by Bodzenta and Donovan in \cite[Theorem 1.2 and Theorem 1.4]{donovan_root_stax}. The authors use this to construct higher spherical functors. A refinement of this result where $X$ is arbitrary was proven by Zhao in \cite[Theorem 1.1]{zhaoyu}.
	\subsection{Acknowledgements}
	I am gratefully indebted to my advisor Jack Hall for suggesting I look into weighted blowups and for his constant guidance, in this project and beyond. I would also like to thank Dougal Davis for numerous helpful suggestions, in particular the proof of Lemma \ref{lemma-key}. Next I would like to thank Will Donovan, for introducing me to the literature on derived categories of root stacks (in particular the work of Bergh-Schn\"urer \cite{ConservDesc}), and for answering my questions about his work \cite{donovan_root_stax}. I have also benefited greatly from comments, suggestions and discussions with Arkamouli Debnath, Sveta Makarova, Fei Peng and David Rydh. 
	
	I am supported by the Australian Commonwealth Government.
	\section{Preliminaries}
	In this section, we collect necessary definitions and results from the literature.
	\subsection{Stacky $\Proj$ and Weighted Blowups}
	
	Let $S$ be an algebraic stack and let $\cal{A} = \bigoplus_{d\in \Z} \cal{A}_d$ be a $\Z$-graded quasi-coherent sheaf of $\O_S$-algebras (note that we are allowing nonzero components in negative degrees). The grading is equivalent to a $\Gm$-action on $\Spec_{\O_S} \cal{A}$ over $S$, by declaring $\Gm$ acts with weight $d$ on $\cal{A}_d$. Write 
	\begin{align*}
		\cal{A}_+ := \bigoplus_{d > 0} \cal{A}_{d}, 
	\end{align*}  
	and 
	\begin{equation*}
		\cal{A}_- := \bigoplus_{d < 0} \cal{A}_{d}.
	\end{equation*}
	\begin{definition}
		We define $\Pproj_{\O_S}(\cal{A})$ to be the algebraic stack \[\Pproj_{\O_S}(\cal{A}) := [(\Spec_{\O_S} \cal{A} \setminus V(\cal{A}_+))/ \Gm]. \] 
	\end{definition}
	Given that we are not requiring $\cal{A}_{-} = 0$, the choice to remove only the vanishing of $\cal{A}_+$ may seem a bit arbitrary. The reason we have opted for this is that we would like to describe blowups as the $\Pproj$ of an extended Rees algebra, which has negative components.
    
	Just as with usual $\Proj$, one can define the line bundles $\O(r)$ corresponding to the shifted modules $\cal{A}(r)$.
	
	We now recall the definition of a weighted blowup. Let $X$ be a fixed algebraic stack. 
	\begin{definition}[{\cite[Definition 3.1.1]{QuekRydh}}]
		A \textit{Rees algebra} on $X$ is a quasi-coherent graded finitely-generated subalgebra $\bigoplus_{n \geq 0} \I_nt^n$ of $\Ox[t]$. 
	\end{definition}
	Note that the collection $\I_\bullet = \{\I_{n}\}_{n\geq 0}$ defines a family of closed substacks $Z_\bullet = \{Z_n\}_{n\geq 0}$. In \cite{QuekRydh}, this data is called a \textit{weighted closed immersion}, we will use the term \textit{weighted centre}: by a \textit{weighted centre} we will understand to mean the weighted centre of some Rees algebra. Thus weighted centres correspond precisely to Rees algebras.

	We will next describe the extension of a Rees algebra.
	\begin{definition}
		Let $\cal{I}_\bullet$ be a weighted centre. The \textit{associated extended Rees algebra} is the graded $\Ox[t^{-1}]$-algebra that is the usual Rees algebra $\bigoplus_{n \geq 0} \I_nt^n$ in non-negative degrees, but is $\Ox$ in all negative degrees. It can be explicitly described as \[\cal{A} := \bigoplus_{n \geq 0} \I_nt^n \oplus \bigoplus_{n < 0}\Ox t^n. \]
	\end{definition}
    
	\begin{definition}
		Let $\I_\bullet$ be a weighted centre with associated extended Rees algebra $\cal{A}$. The \textit{weighted blowup} of $X$ along $\I_\bullet$ is defined to be $\Bl_{\I_\bullet}X := \Pproj_{\Ox}(\cal{A})$. 
	\end{definition}
    \begin{example}
        Suppose $\I_n = \I_1^n$. Then $\Bl_{\I_\bullet}$ is just the usual blowup of $X$ along $V(\I_1)$, by \cite[Remark 3.2.4]{QuekRydh}.
    \end{example}
	Describing the weighted blowup in terms of the extended Rees algebra is advantageous for several reasons: the extended Rees algebra is often easier to describe in practice, and furthermore the exceptional divisor and deformation to the (weighted) normal cone become very easy to define:
	\begin{definition}
		Let $\I_\bullet$ be a weighted centre and let $\tilde{X}$ be its blowup. The \textit{exceptional divisor} is the vanishing locus $E :=V(t^{-1})\subseteq \tilde{X}$. This is an effective Cartier divisor with associated line bundle $\O_{\tilde{X}}(-1)$.
		
		The \textit{weighted normal cone} is defined to be $\Spec_{\Ox}(\bigoplus_{n \geq 0} \I_n/\I_{n+1})$. Note that the exceptional divisor is isomorphic to the $\Pproj$ of the weighted normal cone. Finally, we define the \textit{deformation to the weighted normal cone} to be $\Spec_{\Ox}\cal{A}$.
	\end{definition}
	\begin{remark}
		Unlike with usual blowups, it is not in general true that the following commutative diagram is cartesian:\[\begin{tikzcd}
			E \arrow[r]\arrow[d] & \tilde{X}\arrow[d] \\
			Z_1\arrow[r] & X,
		\end{tikzcd} \] where $Z_1 = V(\cal{I}_1)$.
	\end{remark}
	Next we discuss the various notions of regularity for a Rees algebra. Let us fix some notation: given $f_1,...,f_n\in \Ox$ and $d_1,...,d_n\in \Z_{\geq 1}$, we write $\sum_{i = 1}^n (f_i, d_i)$ for the smallest weighted centre (ordered by degree-wise inclusion) which contains $f_i$ in degree $d_i$.
	\begin{definition}\label{definition-regularity}
		Let $\I_\bullet$ be a weighted centre. We say it is \textit{strictly (quasi, $H_1$, Koszul etc.)-regular} if $\I_\bullet$ is of the form $\sum_{i = 1}^n (f_i, d_i)$ and $f_1,...,f_n$ is a (quasi, $H_1$, Koszul etc.)-regular sequence in $H^0(X, \Ox)$ (see \cite[\href{https://stacks.math.columbia.edu/tag/062D}{Tag 062D}]{stax} and \cite[\href{https://stacks.math.columbia.edu/tag/061M}{Tag 061M}]{stax}). We say it is \textit{(quasi, $H_1$, Koszul etc.)-regular} if smooth-locally at every point on $X$, $\cal{A}$ is strictly (quasi, $H_1$, Koszul etc.)-regular.
	\end{definition}
    It may not immediately clear why requiring the $f_i$ to be an [insert adjective]-regular sequence is of utility, but the key is the following result:
    \begin{lemma}{\cite[Proposition 5.1.2]{QuekRydh}}
        Let $R$ be a ring and let $A =\bigoplus I_nt^n = \sum (f_i, d_i)$ be a Rees algebra. Then $A$ is quasi-regular if and only if each $I_n/I_{n+1}$ is a free $R/I_1$-module.
    \end{lemma}
    In particular, the well-known fact that the exceptional divisor of a (usual) blowup is a projective space bundle over the centre if and only if the centre is quasi-regular generalises. To state the generalisation precisely, we need the following terminology:
    \begin{definition}[{\cite[Definition 2.1.3]{QuekRydh}}]
		Let $X$ be an algebraic stack. A \textit{twisted weighted vector bundle} on $X$ is an algebraic stack $E$ smooth and affine over $X$ with an action of $\Gm$ (also over $X$) such that there exists a smooth cover $X' \rightarrow X$, an $n \geq 0$ and a Cartesian diagram \[\begin{tikzcd}
			\A^{n+1} \times X'\arrow[r, "f"]\arrow[d] & E\arrow[d] \\
			X' \arrow[r]& X,
		\end{tikzcd} \] where $\Gm$ acts linearly on $\A^{n+1}\times X'$ with certain weights $0 <d_0\leq...\leq d_n \in \Z$ and $f$ is $\Gm$-equivariant.
	\end{definition}
	Just as vector bundles correspond to locally free sheaves of finite rank, twisted weighted vector bundles correspond to graded quasi-coherent $\Ox$-algebras that smooth-locally are isomorphic (as graded algebras) to a polynomial algebra, where the variables have positive weights. The correspondence simply takes such an algebra $\cal{A}$ to $\Spec_{\Ox}(\cal{A})$. By abuse of language, we shall also call such algebras twisted weighted vector bundles.
	\begin{definition}\label{def-weighted-stack-bundle}
		A \textit{twisted weighted projective stack bundle} is the $\Pproj$ of a twisted weighted vector bundle.
	\end{definition}
	\begin{proposition}[{\cite[Proposition 5.4.1]{QuekRydh}}]\label{prop-twvb-iff-quasiregular}
		Let $X$ be a quasi-compact algebraic stack and let $\cal{A} = \bigoplus_{n \geq 0} \I_{n}t^n$ be a Rees algebra. Write $Z_n$ for the closed substack defined by $\I_n$. Then $\cal{A}$ is quasi-regular if and only if the weighted normal cone of the blowup of $X$ along $\cal{A}$ is a twisted weighted vector bundle over $Z_1$. In particular, the exceptional divisor will also be a twisted weighted projective stack bundle over $Z_1$ in this case. Moreover, if $\cal{A}$ is strictly quasi-regular then the twisted weighted projective stack bundle over $Z_1$ is a weighted projective stack (i.e. the $\Pproj$ of a graded polynomial ring where the variables have positive weight).
	\end{proposition}
    Let us conclude with some examples:
	\begin{example}\label{eg-h1-regular-extended-rees}
		Suppose $\I_\bullet =(f_1, d_1) + ... + (f_n, d_n)$ is a strictly $H_1$-regular weighted centre with associated extended Rees algebra $\cal{A}$. Then by \cite[Proposition 5.2.2]{QuekRydh}, we have an isomorphism \[\cal{A} \cong \Ox[x_1,...,x_n, t^{-1}]/(f_1-x_1t^{-d_1},...,f_n - x_nt^{-d_n}). \] 
	\end{example}
	\begin{example}\label{eg-extended-rees-affine}
		As a special case of the above example, note that if $X = \A^n_S = \Spec \O_S[x_1,...,x_n]$ over some base algebraic stack $S$ and $\I_\bullet = \sum (x_i, d_i)$, then (forgetting the grading), the extended Rees algebra is the polynomial algebra one more variable: $\cal{A} = \O_S[u_1,...,u_n, t^{-1}]$. In particular, the deformation to the weighted normal cone is also an affine space.
	\end{example}

	\subsection{Triangulated Categories and Conservative Descent}
	We begin by recalling the notions of compactness and generation:
	\begin{definition}[{\cite[Definition 1.7]{neeman_duality}}]
		Let $\sf{T}$ be a triangulated category. An object $\F$ is \textit{compact} if $\Hom(\F,-)$ commutes with coproducts.
	\end{definition}
	\begin{example}
		Let $X$ be a quasi-compact and quasi-separated algebraic stack. If $X$ is concentrated then the compact objects in $\Dqc(X)$ are exactly the perfect complexes (\cite[Lemma 4.4]{perfect_stax}). 
	\end{example}
	\begin{definition}[{\cite[\href{https://stacks.math.columbia.edu/tag/09SI}{Tag 09SI}]{stax}}]\label{def-generators}
		Let $\sf{T}$ be a triangulated category and let $\F$ be an object.  We say $\F$ is a \textit{generator}, if $\Hom(\F, \E[r]) = 0$ for all $r\in \Z$ implies $\E = 0$. More generally, a set $\bf{S}$ of objects in $\sf{T}$ is a \textit{generating set} if $\Hom(\F,  \E[r]) = 0$ for all $\F\in \bf{S}$ and all $r\in \Z$ implies $\E = 0$. We say $\sf{T}$ is \textit{compactly generated} if there is a generating set consisting of compact objects.
	\end{definition}
	\begin{example}\label{eg-gm-characters-generate}
		Let $S$ be an algebraic stack such that $\Dqc(S)$ is compactly generated, and fix $\{\E_\alpha\}$ a set of compact generators. Let $\pi: B\mathbb{G}_{m, S} \rightarrow S$ be the projection map and for each $r\in \Z$ let $\chi_r$ denote the corresponding $\Gm$ character. Then $\{\pi^*\E_\alpha \otimes \chi_r\}$ is a compact generating set for $\Dqc(B\mathbb{G}_{m, S})$. This is easy to see, after noting that $\bigoplus \pi_*(\F \otimes \chi_r) = 0$ implies $\F=0$, for any $\F \in \Dqc(B\bb{G}_{m,S})$.
	\end{example}
   
	\begin{definition}
		Let $\sf{T}$ be a triangulated category. Given an object $\F \in \sf{T}$, we write $\langle \F \rangle$ for the smallest strictly full, saturated (i.e. closed under taking summands) triangulated subcategory of $\sf{T}$ containing $\F$. 
	\end{definition}
	The following simple lemma will turn out useful:
	\begin{lemma}[{\cite[\href{https://stacks.math.columbia.edu/tag/09SK}{Tag 09SK}]{stax}}]\label{lemma-obvious}
		Given any two objects $\E, \F$, the following are equivalent:
		\begin{enumerate}[label = \normalfont(\roman*)]
			\item $\Hom(\F, \E[r]) = 0$ for all $r\in \Z$.
			\item $\Hom(\F', \E) = 0$ for all $\F'\in \langle \F \rangle$.
		\end{enumerate}
	\end{lemma}
	Let us now recall the notion of a semi-orthogonal decomposition:
	\begin{definition}
		Let $\sf{T}$ be a triangulated category and $\sf{T}_1,...,\sf{T}_n$ a sequence of full triangulated subcategories. We say the sequence is \textit{semi-orthogonal} if each $\sf{T}_i$ is right-admissible (i.e. the inclusions admit a right adjoint) and for any $i < j$ and objects $\F_i\in \sf{T}_i$ and $\F_j\in \sf{T}_j$, we have $\Hom(\F_j, \F_i) = 0$. Suppose now $\sf{T}_1,...,\sf{T}_n$ is semi-orthogonal. We say the sequence is \textit{full}, or that it is a \textit{semi-orthogonal decomposition} if the smallest strictly full triangulated subcategory containing the $\sf{T}_i$ is $\sf{T}$. In this case, we write \[ \sf{T} = \langle \sf{T}_1,...,\sf{T}_n \rangle.\]
	\end{definition}
	The following result is our key to showing that a semi-orthogonal sequence is full:
	\begin{lemma}{\cite[Lemma 6.4]{ConservDesc}}\label{lemma-semiorth-full-generator}
		Let $\sf{T}$ be a triangulated category with a generator $\F$ and let $\sf{T}_1,...,\sf{T}_r$ be a semi-orthogonal sequence. Then the sequence is full if and only if their triangulated hull contains $\F$.
	\end{lemma}
	We now recall the theory of conservative descent, as developed in \cite{ConservDesc}. This will allow us to reduce to the affine case. Recall that all functors are derived.
	\begin{definition}[{\cite[Definition 3.3]{ConservDesc}}]\label{def-fourier-mukai}
		Let $S$ be a base algebraic stack and let $X$ and $Y$ be algebraic stacks over $S$. A \textit{Fourier-Mukai transform from $X$ to $Y$ over $S$} is a tuple $(K, \cal{P}, p,q)$, where we have the commutative diagram of algebraic stacks: \[\begin{tikzcd}
			&K \arrow[ld, "p"] \arrow[dr, "q"]\\
			X\arrow[rd] && \arrow[dl]Y\\
			&S,
		\end{tikzcd} \] and $\cal{P}$ is a perfect complex on $K$, such that the following hold:
		\begin{enumerate}[label = \normalfont(\roman*)]
			\item\label{condition-1} The maps $p,q$ are perfect (see \cite[\href{https://stacks.math.columbia.edu/tag/0685}{Tag 0685}]{stax} - which generalises easily to stacks), proper and concentrated (see \cite[\S2]{perfect_stax}).
			\item\label{condition-2} The object $q^! \O_Y$ is perfect (where $q^!: \Dqc(Y)\rightarrow \Dqc(K)$ is the right adjoint to $Rq_*$ - which exists by \cite[Theorem 4.14]{perfect_stax}).
			\item\label{condition-4} Given any flat $u: Y' \rightarrow Y$ with Cartesian square \[\begin{tikzcd}
				K' \arrow[r]\arrow[d, "q'"] \arrow[r, "v"]& K \arrow[d, "q"]\\
				Y' \arrow[r, "u"] & Y,
			\end{tikzcd} \] the natural transformations $$v^*q^!\rightarrow (q')^!u^*$$ and $$(q')^!\O_Y \otimes L(q')^*(-)\rightarrow (q')^!(-)$$ are isomorphisms.
		\end{enumerate}
	\end{definition}
	\begin{remark}
	    We should really be using $q^\times$ instead of $q^!$. However, in all situations dealt with in the paper, the two functors coincide.
	\end{remark}
	Given a Fourier-Mukai transform, we have an associated functor $\F \mapsto q_*(p^*\F \otimes \cal{P})$ which admits a right adjoint $\G \mapsto p_*(q^!\G\otimes \cal{P}^\vee)$. Moreover, 
	given a flat morphism $S' \rightarrow S$, there is an obvious notion of flat base-change.
	\begin{theorem}[{\cite[Theorem 6.1]{ConservDesc}}]\label{thm-cons-desc}
		Let $S$ be a base algebraic stack. Let $X_1,...,X_n$ and $Y$ be algebraic stacks over $S$ and let $\Phi_i: \Dqc(X_i)\rightarrow \Dqc(Y)$ be (the associated functors of) Fourier-Mukai transforms over $S$. Let $S' \rightarrow S$ be a faithfully flat map and denote the base-changes by $X_i', Y', \Phi_i'$ etc. 
		\begin{enumerate}[label = \normalfont(\roman*)]
			\item If $\Phi_i'$ is fully faithful, then so is $\Phi_i$.
			\item Assuming all the $\Phi_i'$ are fully faithful, if $\im \Phi_1',...,\im \Phi_n'$ is a semi-orthogonal sequence in $\Dqc(Y')$, then so is $\im \Phi_1,...,\im \Phi_n$ in $\Dqc(Y)$. 
			\item If $\im \Phi_1',...,\im \Phi_n'$ is a semi-orthogonal decomposition of $\Dqc(Y')$, then $\im \Phi_1,...,\im \Phi_n$ is a semi-orthogonal decomposition of $\Dqc(Y)$. 
		\end{enumerate}	
	\end{theorem}
	\begin{remark}
		Note that our definition differs from {\cite[Definition 3.3]{ConservDesc}} slightly, in that {\cite[Definition 3.3]{ConservDesc}} further requires $Rp_*$ and $Rq_*$ preserve perfect and pseudo-coherent complexes. However, this requirement is used only to show that the induced functor preserves perfect and (locally bounded) pseudo-coherent complexes respectively. The advantage of this requirement is that Theorem \ref{thm-cons-desc} works with the categories of perfect and (locally bounded) pseudo-coherent complexes in place of $\Dqc$ - however its absence does not affect the corresponding statement for $\Dqc$ itself. 
		
		Since it is not so clear at the time of writing that our relevant Fourier-Mukai transforms preserve pseudo-coherent complexes in the non-noetherian setting (the main issue being the absence of a theory of \enquote{approximation by perfect complexes}), we have opted to remove this requirement, and instead we will manually deduce the corresponding statements for $\Perf$ etc.
	\end{remark}
	\section{Some cohomology computations}
	The goal of this section is to first prove the below result. 
	\begin{proposition}\label{prop-sga6}
		Let $X$ be an algebraic stack and let $\pi:\tilde{X} \rightarrow X$ be the weighted blowup of a Koszul-regular weighted centre. Then the pullback map $L\pi^*: \Dqc(X) \rightarrow \Dqc(\tilde{X})$ is fully faithful.
	\end{proposition}

	When $X$ is a smooth variety and $\tilde{X}$ is the (usual) blowup along a smooth subvariety, the typical argument (e.g. in \cite[III Proposition 3.4]{Hart}) proceeds by using the theorem on formal functions to show $R^i\pi_*(\O_{\tilde{X}}) = 0$ for all $i > 0$ and appealing to the normality of $X$ to establish the isomorphism of $\Ox \rightarrow R^0\pi_*\O_{\X}$. When $X$ is an arbitrary scheme and $\X$ is the blowup of a Koszul-regular closed subscheme (in the sense of  \cite[\href{https://stacks.math.columbia.edu/tag/063J}{Tag 063J}]{stax}), this result is established in \cite[Lemme 3.5]{sga6} by a descending induction argument. We will give a different, arguably more intuitive proof.

	We begin by recalling the following standard result:
	\begin{lemma}\label{lemma-cohomology-punctured-affine}
		Let $S$ be an algebraic stack. Let $j: \A^{n+1}_S\setminus \{0\}\hookrightarrow \A^{n+1}_S$ denote the inclusion and $i: Z =0:= S \hookrightarrow \A^{n+1}_S$ denote the complement. Denote also the coordinates of $\A^{n+1}$ by $x_0,...,x_n$. Then \[R\Gamma_Z(\O_{\A^{n+1}}) = \frac{1}{x_0...x_n}\O_S[{x_0}^{-1},...,x_n^{-1}] [-n-1]. \] Moreover, the map $R\Gamma_Z(\O_{\A^{n+1}})\rightarrow R\Gamma(\O_{\A^{n+1}})$, considered as a map in $\Dqc(S)$, is zero, and consequently \[Rj_*(\O_{ \A^{n+1}_S\setminus \{0\}}) = \O_S[x_0,...,x_n] \oplus \frac{1}{x_0...x_n}\O_S[{x_0}^{-1},...,x_n^{-1}] [-n].\]
	\end{lemma}
	\begin{proof}
		For $S = \Spec \Z$ it is well-known. For the general case, base-change along $S \rightarrow \Spec \Z$.
	\end{proof}
	\begin{lemma}\label{lemma-cohomology-line-bundles}
		Let $S$ be an algebraic stack and let \[f:\scr{P}(d_0,...,d_n) = \Pproj \O_S[x_0,...,x_n]\rightarrow S\] be the weighted projective stack over $S$ with weights $1 \leq d_0 \leq...\leq d_n$. Then: \[Rf_*(\O(r)) = \O_S[x_0,...,x_n]_{ r} \oplus \left(\frac{1}{x_0...x_n}\O_S[{x_0}^{-1},...,x_n^{-1}]_{r}\right)[-n]. \] 
	\end{lemma}
	
	\begin{proof}
		This follows immediately from Lemma \ref{lemma-cohomology-punctured-affine} and the fact that $B\Gm$ is cohomologically affine. 
	\end{proof}
    

	\begin{lemma}\label{lemma-key}
		Let $A$ be a ring and let $f_0,...,f_n$ be a Koszul-regular sequence in $A$. Then for any tuple $(d_0,...,d_n)$ of positive integers, the following sequence \[f_0-x_0s^{d_0},...,f_n - x_ns^{d_n} \] in the polynomial ring $A[s,x_0,...,x_n]$ is also Koszul-regular.
	\end{lemma}
	\begin{proof}
		By \cite[\href{https://stacks.math.columbia.edu/tag/0669}{Tag 0669}]{stax}, it follows that $f_0,...,f_n, s$ is a Koszul-regular sequence in $A[s,x_0,...,x_n]$, and hence $f_0-x_0s^{d_0},...,f_n - x_ns^{d_n} , s$ is also Koszul-regular, since it generates the same ideal. Let us write \[K_\bullet := K_\bullet(f_0-x_0s^{d_0},...,f_n - x_ns^{d_n}; A[s,x_0,...,x_n] ) \] and \[K_\bullet' := K_\bullet(f_0-x_0s^{d_0},...,f_n - x_ns^{d_n} , s; A[s,x_0,...,x_n]). \] We have a distinguished triangle in $\mathrm{D}(A[s,x_0,...,x_n])$ of Koszul complexes: \[K_\bullet \xrightarrow{\times s}K_\bullet\rightarrow K_\bullet' \rightarrow K_\bullet[1], \] and by assumption, $H_r(K'_\bullet) = 0$ for all $r>0$. In particular, multiplication by $s$ induces an isomorphism on $H_r(K_\bullet)$ for all $r > 0$. 
		
		Now note that if we invert $s$, then the change-of-variables $x_i \mapsto f_i - x_is^{d_i}$ becomes invertible, and hence there is a ring automorphism of $A[s^{\pm1},x_0,...,x_n]$ taking $f_0-x_0s^{d_0},...,f_n - x_ns^{d_n}$ to the Koszul-regular sequence $x_0,...,x_n$. Thus for any $r > 0$ we have \[H_r(K_\bullet)\otimes A[s^{\pm 1},x_0,...,x_n]= H_r(K_\bullet \otimes A[s^{\pm 1},x_0,...,x_n]) = 0, \] and in particular every element of $H_r(K_\bullet)$ is annihilated by some power of $s$. But as established, multiplication by $s$ is an isomorphism on $H_r(K_\bullet)$, and hence we must have $H_r(K_\bullet) = 0$ as desired.
	\end{proof}
    \begin{lemma}\label{lemma: local-coho-calc-Rees}
        Let $R$ be a ring and let $(f_0, d_0)+...+(f_n, d_n)$ be a strictly Koszul-regular weighted centre, with extended Rees algebra $\cal{A} = R[t^{-1}, x_0,...,x_n]/(f_i-x_it^{-d_i})$. Letting $Z = V(x_0,...,x_n)\subseteq \Spec \cal{A}$, we have \[R\Gamma_Z(\O_{\Spec \cal{A}}) = \frac{1}{x_0...x_n}R[t^{-1},x_0^{- 1},...,x_n^{- 1}] /(f_i-x_it^{-d_i})[ -n-1].  \]
    \end{lemma}
    \begin{proof}
        By \cite[\href{https://stacks.math.columbia.edu/tag/0ALZ}{Tag 0ALZ}]{stax}, and Lemma \ref{lemma-cohomology-punctured-affine}, we have \[R\Gamma_Z(\O_{\Spec \cal{A}}) = \frac{1}{x_0...x_n}R[t^{-1},x_0^{-1},...,x_n^{-1}][-n-1] \ltensor_{R[t^{-1},x_0,...,x_n]}  R[t^{-1}, x_0,...,x_n]/(f_i-x_it^{-d_i}).  \] Thus it suffices to show that $f_0-x_0t^{-d_0},...,f_n-x_nt^{-d_n}$ is a Koszul-regular sequence for the $R[t^{-1},x_0,...,x_n]$-module $\frac{1}{x_0...x_n}R[t^{-1},x_0^{-1},...,x_n^{-1}]$.
        
       Set $M := \frac{1}{x_0...x_n}R[t^{-1},x_0^{-1},...,x_n^{-1}]$ and let us consider the Koszul complex \[K_\bullet = K_\bullet(f_0-x_0t^{-d_0},...f_n-x_nt^{-d_n}; M). \] We define a filtration on $K_\bullet$: for $p\in \Z$, we let $F^pK_\bullet$ consist of the subcomplex of $R[t^{-1}]$-modules spanned by the monomials $x_{I}^{a_{I}}$, where $|a_I| \geq p$. So $F^pK_\bullet = 0$ for all $p \geq -n$. It is clear the filtrations respect the differentials, making $K_\bullet$ a filtered complex.

       Now let us consider the associated spectral sequence, which has $E_0$-page  $E_0^{p,q} = \operatorname{gr}^pK_{-p-q}$. By construction, the $\operatorname{gr}^pK_{-p-q}$ are free $R[t^{-1}]$-modules, and moreover, since the $x_i$ are killed, each $E_0^{p, \bullet}[-p]$ is simply the Koszul complex of the $R[t^{-1}]$-Koszul-regular sequence (by Lemma \ref{lemma-key}) $f_0-x_0t^{-d_0},..., f_n-x_nt^{-d_n}$ tensored with the free module spanned by the monomials $x_I^{a_I}$ with $|a_I| = p$. Thus $E_1^{p,q} = 0$ unless $q = -p$, and hence the spectral sequence degenerates here. Since our filtration is bounded above and exhaustive the spectral sequence converges, which concludes the proof.
    \end{proof}
    
    \begin{corollary}\label{cor: coho-calc-Rees}
        Let $R$ be a ring and let $\cal{A} = (f_0,d_0)+...+(f_n, d_n)$ be a strictly Koszul-regular weighted centre with extension $\cal{A}$. Let $f:\tilde{X} = \Pproj \cal{A}\rightarrow X = \Spec R$ denoted the weighted blowup. Then \[R\Gamma(\tilde{X}, \O(r)) = \cal{A}_r \oplus
            \left(\frac{1}{x_0...x_n}R[t^{-1},x_0^{- 1},...,x_n^{- 1}] /(f_i-x_it^{-d_i})_r\right) [-n]. \]
    \end{corollary}
    \begin{proof}
        Direct from Lemma \ref{lemma: local-coho-calc-Rees}, Lemma \ref{lemma-cohomology-punctured-affine} and that $B\Gm$ is cohomologically affine.
    \end{proof}

    \begin{proof}[Proof of Proposition \ref{prop-sga6}]
		 By the projection formula, it suffices to show $R\pi_*\O_{\tilde{X}} = \Ox$. The question is local, so we may assume $X$ is affine and the Rees algebra is strictly Koszul-regular. The result is then immediate from Corollary \ref{cor: coho-calc-Rees}.
    \end{proof}

	\section{Affine Case}
	In this section, we prove a local special case of Theorem \ref{thm-main} and Theorem \ref{thm: shriek}. Along the way, we establish a weighted analogue of Beilinson's classical decomposition \cite[Theorem]{Beilinson}.
	
	
	\begin{lemma}\label{lemma-lb-are-generators}
		Let $S$ be an algebraic stack such that $\Dqc(S)$ is compactly generated, and fix $\{\E_\alpha\}$ a generating set of compact objects. Let $U$ be a quasi-affine stack with a $\Gm$-action over $S$ and let $X = [U/\Gm]$ denote the resulting quotient stack. Then $\{L\pi^*\E_\alpha(r)\}$ forms a compact generating set for $\Dqc(X)$. Here $\F(r)$ refers to $\F$, twisted by the line bundle $\O(r)$.
	\end{lemma}
	\begin{proof}
		Let us factor $\pi$ as below: \[\begin{tikzcd}
			X\arrow[d, "f"] \arrow[rd, "\pi"]\\
			 B \Gm\times S \arrow[r, "g"] & S
		\end{tikzcd} \] Letting $\F\in \Dqc(X)$ and supposing $R\Hom(L\pi^*\E_\alpha(r), \F) = 0$ for each $\alpha, r$, by adjunction we have \[ 0 = R\Hom(L\pi^*\E_\alpha(r), \F) = R\Hom(Lf^*Lg^*\E_\alpha(r), \F) = R\Hom(Lg^*\E_\alpha(r), Rf_*(\F)).\] Thus by Example \ref{eg-gm-characters-generate} we have $Rf_*\F=0$. But $f$ is quasi-affine, thus $\F=0$.
	\end{proof}
	\begin{lemma}\label{lemma-koszul-beilinson}
		Let $R$ be a $\Z$-graded ring and let $x_0,...,x_n$ be a Koszul-regular sequence of positively graded homogeneous elements, where $x_i$ has degree $d_i > 0$. Let $X = [(\Spec R \setminus V(x_0,...,x_n))/ \Gm]$. If $R\Hom(\O(d), \F) = 0$ for all $1-\sum d_i\leq d \leq 0$, then $\F = 0$, In other words, the object $\O(1-\sum d_i)\oplus...\oplus\O(-1)\oplus\O$ is a generator for $\Dqc(X)$. In particular, $\Dqc(X)$ is generated by a single perfect complex.
	\end{lemma}
	\begin{proof}
		Observe we have the following exact sequence: \begin{equation}\label{eqn-alternative-to-beilinson}
			0 \rightarrow \O(-\sum d_i) = \wedge^{n+1 }( \bigoplus_{i = 0}^n \O(-d_i) )\rightarrow ... \rightarrow \bigoplus_{i = 0}^n \O(-d_i) \rightarrow \O \rightarrow 0.
		\end{equation} 
		Indeed, this is just the (appropriately graded) Koszul resolution associated to the Koszul-regular sequence $x_0,...,x_n$, whose vanishing locus is removed by assumption, hence the sequence is exact. In particular, this tells us $\O$ is quasi-isomorphic to a complex consisting of direct sums of sheaves of the form $\O(-r)$ for $-\sum d_i \leq -r \leq -1$ and so $\O(1)$ is quasi-isomorphic to a complex consisting of direct sums of sheaves of the form $\O(-r)$ for $1-\sum d_i \leq -r \leq 0$ (by twisting (\ref{eqn-alternative-to-beilinson})). Applying this inductively, this means every $\O(s)$ can be obtained by a finite number of extensions of the $\O(1-\sum d_i),...,\O(-1), \O$, or in other words $\O(s)\in \langle \O(1-\sum d_i)\oplus...\oplus\O(-1)\oplus\O \rangle$ for all $s\in \Z$. We win, by Lemma \ref{lemma-lb-are-generators} and Lemma \ref{lemma-obvious}.
	\end{proof}
	
	\begin{corollary}\label{cor-weighted-beilinson-affine}
		Let $X = \Spec R$ be an affine scheme and let \[\pi:\scr{P}=\Pproj R[x_0,...,x_n]\rightarrow X\] be the weighted projective stack over $R$ with weights $1 \leq d_0 \leq...\leq d_n$. Then there is a semiorthogonal decomposition \[\Dqc(\scr{P}) = \langle \O_{\scr{P}}(1-\sum d_i)\otimes \pi^*\Dqc(X) ,...,\O_{\scr{P}}(-1)\otimes \pi^*\Dqc(X), \pi^*\Dqc(X)\rangle, \] where $\O_{\scr{P}}(r)\otimes\pi^*\Dqc(X)$ is the full triangulated subcategory consisting of objects of the form $\O_{\scr{P}}(r)\otimes\pi^*(\F)$ where $\F\in \Dqc(X)$.
	\end{corollary}
	\begin{proof}
		First note that, by the projection formula and Lemma \ref{lemma-cohomology-line-bundles} (ii), the functor $\pi^*: \Dqc(\Spec R) \rightarrow \Dqc(\scr{P})$ (and hence $\pi^*(-)\otimes \O(r)$, since twisting is an equivalence.) is fully faithful. Moreover each $\O_{\scr{P}}(-r)\otimes \pi^*\Dqc(X)$ is right admissible (i.e. the inclusion has a right adjoint); indeed, the functor \[\F \mapsto \O_{\scr{P}}(-r)\otimes\pi^*(R\Gamma(\F \otimes \O_{\scr{P}}(r))) \] is the right adjoint. 
		
		Now we show that this is indeed a semi-orthogonal decomposition. Indeed, if $1-\sum d_r \leq -i < -j \leq 0$ it follows
		\begin{align*}
			R\Hom(\pi^*\F \otimes \O(-j),\pi^*\G \otimes \O(-i) ) &= R\Hom(\pi^*\F, \pi^*\G \otimes \O(-i + j))  \\&= R\Hom(\F, R\Gamma(\O(-i+j)) \otimes \G)
		\end{align*} by the projection formula. Since $R\Gamma(\O(-i+j)) = 0$ by Lemma \ref{lemma-cohomology-line-bundles}, it follows that \[R\Hom(\pi^*\F \otimes \O(-j),\pi^*\G \otimes \O(-i) ) = 0. \] In particular the sequence is semi-orthogonal.	The fact that it is full now follows directly from Lemma \ref{lemma-koszul-beilinson} and Lemma \ref{lemma-semiorth-full-generator}.
	\end{proof}
	\begin{remark}
		Of course, (modulo boundedness and coherence conditions) this generalises the well-known semi-orthogonal decomposition of a projective space (due to Beilinson \cite{Beilinson} over a field and Orlov \cite{OrlovThesis} in the general case) to the case of a weighted projective stack. Over a field this is done in \cite[Theorem 2.2.4]{WeightedBeilinsonPre}. 
	\end{remark}
	\begin{proposition}\label{prop-main-affine}
		Let $X = \Spec R$ be an affine scheme and let $\I_\bullet = \sum_{i = 0}^n (f_i, d_i)$ be a strictly Koszul-regular weighted centre on $X$ with associated extended Rees algebra $\cal{A}$. Write $Z_r$ for $V(\cal{I}_r)$, the closed subscheme corresponding to $\cal{I}_r$. Let $\tilde{X} $ be the weighted blowup of $\I_\bullet$ and let $E$ be the exceptional divisor, which is the weighted projective stack $E =\scr{P}_{Z_1}(d_0,...,d_n)$ over $Z_1$ with $1 \leq d_0 \leq...\leq d_n$ by Proposition \ref{prop-twvb-iff-quasiregular}. We have a commutative diagram: \[\begin{tikzcd}
			E \arrow[r, hook,"j"]\arrow[d, "p"] & \tilde{X} \arrow[d, "\pi"] \\
			Z _1\arrow[r,hook, "i"] & X.
		\end{tikzcd} \] Then there is a semi-orthogonal decomposition: \[\Dqc(\tilde{X}) = \langle j_*(p^*\Dqc(Z_1)\otimes\O_E(1- \sum d_i)),...,j_*(p^*\Dqc(Z_1)\otimes\O_E(-1)), L\pi^*\Dqc(X) \rangle. \]
	\end{proposition}
	\begin{proof}
		First note that all relevant functors from $\Dqc(Z_1)$ and $\Dqc(X)$ to $\Dqc(\X)$ are fully faithful - indeed $\pi^*$ is fully faithful by Proposition \ref{prop-sga6} and $p^*$ is fully faithful by Lemma \ref{lemma-cohomology-line-bundles}. While $j_*$ will in general not be fully faithful, it is fully faithful on each subcategory $p^*\Dqc(Z_1)\otimes\O_E(r)$, indeed the argument is identical to the corresponding argument in the proof of \cite[Theorem 6.9]{ConservDesc}. Ditto for the argument that the sequence is semi-orthogonal (here Corollary \ref{cor-weighted-beilinson-affine} is needed). They all clearly admit right adjoints.

		We just need to show that the sequence is full. Now we have the description \[\tilde{X} = [\Spec R[t^{-1}, x_0,...,x_n]/(f_0 - x_0t^{-1},...,f_n - x_nt^{-1}) \setminus V(x_0,...,x_n) / \Gm]. \] Thus by Lemma \ref{lemma-koszul-beilinson} and Lemma \ref{lemma-semiorth-full-generator}, it suffices to show that $\O_{\tilde{X}}(r)$ for $1-\sum d_i\leq r \leq 0$ is contained in the triangulated hull of the subcategories $j_*(p^*\Dqc(Z_1)\otimes\O_E(1- \sum d_i)),...,j_*(p^*\Dqc(Z_1)\otimes\O_E(-1)), \pi^*\Dqc(X)$. Let us denote this triangulated hull $\sf{T}$. Clearly $\O_{\X}\in \sf{T}$. Now for $1 - \sum d_i \leq r \leq -1$, by assumption $j_*\O_E(r)\in \sf{T}$. But $$j_*\O_E(r) = [\O_{\X}(r+1)\rightarrow \O_{\X}(r)],$$ so by a downwards induction argument, we see that $\O_{\X}(r)\in \sf{T}$ as well, as desired.
	\end{proof}	
    
    \begin{example}[{\cite[Example 5.4.4]{QuekRydh}}]\label{eg: fkn-notation}
        Working over a ring $k$, consider $X = \Spec k[x,y]$. Letting $\tilde{X}$ denote the blowup of $X$ along the weighted centre $(x,3) + (y,2)$, i.e. $$\tilde{X} = \Pproj k[x,y][u,v, t^{-1}]/(x-ut^{-3}, y-vt^{-2})= \Pproj k[u,v, t^{-1}]$$ where $\deg u = 3, \deg v = 2$, we denote by $\tilde{X}(\sqrt[3]{u}, \sqrt[2]{v})$ the blowup of $\tilde{X}$ along the centre $(u,3)$ followed by $(v,2)$ (i.e. we are taking two root stacks). Then we may also obtain $\tilde{X}(\sqrt[3]{u}, \sqrt[2]{v})$ by blowing up $X(\sqrt[3]{x}, \sqrt[2]{y}) = \Bl_{(x,3)}\Bl_{(y,2)}(X)$ along the (unweighted) intersection of the exceptional divisors. Thus we have the following commutative diagram \[\begin{tikzcd}[row sep = huge]
            \tilde{X}(\sqrt[3]{u}, \sqrt[2]{v}) \arrow[r] \arrow[d,swap,  "\text{usual blowup}"] & \tilde{X} \arrow[d, "\text{weighted blowup}"] \\
             X(\sqrt[3]{x}, \sqrt[2]{y})\arrow[r] & X,
        \end{tikzcd} \]where the horizontal arrows are both sequences of root stacks. In particular, we have exhibited $\tilde{X}(\sqrt[3]{u}, \sqrt[2]{v})$ as a sequence of weighted blowups of $X$ in two different ways, and thus $\Dqc(\tilde{X}(\sqrt[3]{u}, \sqrt[2]{v}))$ should admit two different semi-orthogonal decompositions. 

        
        Indeed, on one hand (going through the weighted blowup) we have a decomposition with terms 
        \begin{align*}
            \Dqc(\tilde{X}(\sqrt[3]{u}, \sqrt[2]{v})) = \langle&\Dqc(\pt),\Dqc(\A^1),\Dqc(\pt),\Dqc(\A^1),&\text{rooting }u=0\\ &\Dqc(\pt),\Dqc(\pt), 
            \Dqc(\A^1),&\text{rooting }v=0\\ &\underbrace{\Dqc(\Spec k),\dots \Dqc(\Spec k)}_{5 \text{ copies}},\Dqc(X) \rangle &\text{the weighted blowup}. 
        \end{align*}
        On the other hand (going through the usual blowup) we have a decomposition with terms 
        \begin{align*}
            \Dqc(\tilde{X}(\sqrt[3]{u}, \sqrt[2]{v})) = \langle& 
            \Dqc(B\mu_6) = \Dqc(\pt)^{\oplus 6},&\text{ the usual blowup}\\&\Dqc(\pt),\Dqc(\A^1), \Dqc(\pt),\Dqc(\A^1) &\text{rooting }x=0\\ 
            &\Dqc(\pt),\Dqc(X) \rangle &\text{rooting }y=0. 
        \end{align*}
        Repeated application of \cite[Theorem 1.2]{donovan_root_stax} whilst keeping track of the twists will relate these two decompositions. To avoid an explosion of notation, we leave the details to the reader.

        
    \end{example}
    Next, we prove Theorem \ref{thm: shriek} in the affine case.
    \begin{definition}
        Let $R$ be a ring and let $\sum (f_i, d_i)$ be a strictly Koszul-regular weighted centre, with extended Rees algebra $\cal{A} = R[t^{-1}, x_0,...,x_n]/(f_i-x_it^{-d_i})$. Let $f: \tilde{X} = \Pproj \cal{A} \rightarrow X = \Spec R$ be the weighted blowup. We define the \textit{trace map} associated to $f$ to be the isomorphism $$\int_f: R\Gamma(\tilde{X}, \O(1-\sum d_i)) = H^0(\tilde{X}, \O(1-\sum d_i)) \xrightarrow{\sim} R$$sending $t^{1-\sum d_i}$ to 1.
    \end{definition}
     \begin{lemma}\label{lemma: duality-compact-generators}
        Let $Rf_*: \sf{T}' \rightarrow \sf{T}$ be a coproduct-preserving triangulated functor between two triangulated categories which contain all coproducts, and suppose $\sf{T'}$ is compactly generated. Let $t\in T$, and suppose we have a natural transformation \[\eta(-): \Hom_{\sf{T}'}(-, \E) \rightarrow \Hom_{\sf{T}}(Rf_*(-), t) \] for some object $\E$ of $\sf{T'}$. Now suppose there is a set $\Omega$ of compact generators of $\sf{T}'$ closed under shifts such that $\eta(\L)$ is an isomorphism for all $\L\in \Omega$. Then $\eta$ is an isomorphism.
    \end{lemma}
    \begin{proof}
        By Brown representability (\cite[Theorem 4.1]{neeman_duality}), $Rf_*$ admits a right adjoint $f^!$. Set $\F = f^!(t)$. Then by Yoneda's lemma, we have a map $\E \rightarrow \F$, which we complete to a triangle \[\E \rightarrow \F \rightarrow \cal{G} \rightarrow \E[1].\] Applying the cohomological functor $\Hom(\L, -)$ tells us $\Hom(\L, \G) = 0$ for all $\L\in \Omega$, whence $\G=0$.
    \end{proof}
    
    \begin{proposition}\label{prop: shriek-affine}
        Let $R$ be a ring and let $(f_0, d_0)+...+(f_n, d_n) =: I_\bullet$ be a strictly Koszul-regular weighted centre, with extended Rees algebra $\cal{A} = R[t^{-1}, x_0,...,x_n]/(f_i-x_it^{-d_i}$. Let $f: \tilde{X} = \Pproj \cal{A} \rightarrow X = \Spec R$ be the weighted blowup. Let $f^!: \Dqc(X) \rightarrow \Dqc(\tilde{X})$ denote the right adjoint of $Rf_*$ (which exists by Brown representability and the observation that $\tilde{X}$ is clearly concentrated). Then the trace map induces an isomorphism  \[f^!(-) \iso f^*(-) \otimes \O(1-\sum d_i) \] of functors.
    \end{proposition}
    \begin{proof}
        Lemma \ref{lemma-lb-are-generators} and Corollary \ref{cor: coho-calc-Rees} in combination with \cite[Theorem 5.1]{neeman_duality} tell us that $f^!$ preserves coproducts, and thus by \cite[Theorem 5.4]{neeman_duality} (it is easy to check the proof applies in this situation) we know $f^!$ is of the form $f^! = f^* \otimes \omega$.

        The trace map induces a morphism of functors \[\eta: R\Hom_{\tilde{X}}(- , \O_{\tilde{X}}(1-\sum d_i )) \rightarrow R\Hom_R(R\Gamma(\tilde{X}, -), R). \] By Lemmas \ref{lemma-lb-are-generators} and \ref{lemma: duality-compact-generators}, it suffices to show that verify that the maps $\eta(\O_{\tilde{X}}(r))$ are isomorphisms. For $r= 1-\sum d_i$ this follows from Corollary \ref{cor: coho-calc-Rees}, and letting $j: E \hookrightarrow \tilde{X}$ denote the inclusion of the exceptional divisor, using the triangle \[\O_{\tilde{X}}(r+1) \rightarrow \O_{\tilde{X}}(r) \rightarrow j_*\O_E(r) \rightarrow \O_{\tilde{X}}(r+1)[1] \] and the five-lemma we reduce to checking that $\eta(\O_E(r))$ is an isomorphism for all $r$.

        By \cite[Theorem 4.14]{perfect_stax}, $j^!$ also preserves coproducts, and we can calculate \[j^!\O_{\tilde{X}} = R\Hom(j_*\O_E, \O_{\tilde{X}}) = [\O_{\tilde{X}} \rightarrow \O_{\tilde{X}}(-1)] = \O_E(-1)[-1]\] (where the final isomorphism should be thought of as an isomorphism of dg-rings). The counit map $$j_*j^!\O_{\tilde{X}} = j_*\O_E(-1)[-1] \rightarrow \O_{\tilde{\X}}$$ is simply the map induced by the short exact sequence \[ 0 \rightarrow \O(1) \rightarrow \O \rightarrow j_*\O_E \rightarrow 0. \] Thus we have induced maps \[R\Hom_E(\O_E(r), \O(-\sum d_i)[-1]) \rightarrow R\Hom(R\Gamma(E, \O_E(r)), R). \] Using Lemma \ref{lemma-cohomology-line-bundles} and exhaustive checking (noting that $R\Gamma(E, \O_E(r))$ is a split complex of free $R/f_0,...,f_n$-modules), we may verify they are isomorphisms, which concludes the proof.

    \end{proof}
	\section{Reduction to the Affine Case}
	We begin this section by proving Theorem \ref{thm-main}. We apply the theory of conservative descent developed in \cite{ConservDesc} to reduce to the affine case. Let us fix the situation:
	\begin{sitch}\label{sitch}
		Fix $(d_0,...,d_n)$ a tuple of positive integers. Let $X$ be a quasi-compact algebraic stack and let $\I_\bullet$ be a Koszul-regular weighted centre on $X$, that smooth-locally is of the form $\sum_{i = 0}^n (f_i, d_i)$, let $\X$ denote the blowup and let $E$ denote the exceptional divisor. Write $Z_i = V(\cal{I}_i)$. We have a commutative diagram: \[\begin{tikzcd}
			E \arrow[r, hook,"j"]\arrow[d, "p"] & \tilde{X} \arrow[d, "\pi"] \\
			Z _1\arrow[r,hook, "i"] & X.
		\end{tikzcd} \] 
	\end{sitch}
	\begin{lemma}\label{lemma-fourier-mukai-1}
		In Situation \ref{sitch}, the quadruple $(\tilde{X}, \O_{\tilde{X}}, \pi, \id)$ is a Fourier-Mukai transform (see Definition \ref{def-fourier-mukai}) over $X$.
	\end{lemma}
	\begin{proof}
		Only condition \ref{condition-1} for $\pi$ needs to be checked, and it can be checked smooth-locally, so we may assume $X$ is affine and the centre is strictly Koszul-regular, say $\sum_{i = 0}^n (f_i, d_i)$. Now $\pi$ is proper by \cite[Proposition 1.6.1]{QuekRydh} and concentrated by \cite[Theorem C]{alg_gps_compact} and \cite[Lemma 1.1.2]{QuekRydh}. The fact that $\pi$ is perfect follows from Lemma \ref{lemma-key}. 
	\end{proof}
	\begin{lemma}\label{lemma-fourier-mukai-2}
		In Situation \ref{sitch} the quadruple $(E, \O_E(r), p, j)$ is a Fourier-Mukai transform over $X$ for all $r\in \Z$.
	\end{lemma}
	\begin{proof}
		Condition \ref{condition-1} can be checked smooth-locally on $X$, so we may assume that $X$ is affine and the weighted centre is strictly Koszul-regular. Firstly, $j$ is a closed immersion so it is clearly proper, and concentrated. Moreover, it is an effective Cartier divisor, so it is also perfect by by \cite[\href{https://stacks.math.columbia.edu/tag/068C}{Tag 068C}]{stax}. Now $p$ is a weighted projective stack, by Proposition \ref{prop-twvb-iff-quasiregular}, so clearly it is smooth (hence perfect by \cite[\href{https://stacks.math.columbia.edu/tag/07EN}{Tag 07EN}]{stax}) and proper. Finally, it is concentrated by \cite[Theorem C]{alg_gps_compact} and \cite[Lemma 1.1.2]{QuekRydh}.
		
		Using \cite[Theorem 4.14 (2)]{perfect_stax}, we can explicitly calculate $$j^!\O_{\X} = R\Hom(j_*\O_E, \O) = [\Ox \rightarrow \Ox(-1)] = \O_E(-1)[-1],$$ (which implies \ref{condition-2}. In fact more generally, if $i:D \hookrightarrow Y$ is an effective Cartier divisor, then $i^! \O_Y = \O_D(D)[-1]$.
		
		
		Finally, \ref{condition-4} follows directly from the above observation and \cite[Theorem 4.14 (4)]{perfect_stax}.
	\end{proof}
	\begin{proof}[Proof of Theorem \ref{thm-main}]
		Suppose firstly $\D = \Dqc$. Since $X$ is quasi-compact, there exists a smooth cover by an affine scheme. The result then follows from Lemma \ref{lemma-fourier-mukai-1}, Lemma \ref{lemma-fourier-mukai-2}, Theorem \ref{thm-cons-desc} and Proposition \ref{prop-main-affine}.        
        
		For the other cases, let us introduce an abuse of language: given a triangulated functor $F: \Dqc(Y_1)\rightarrow \Dqc(Y_2)$, we will say $F$ \textit{preserves} $\D$ if $F|_{\D(Y_1)}: \D(Y_1)\rightarrow \Dqc(Y_2)$ factors through $\D(Y_2)$. We now claim that in all other cases, the (induced functors of the) Fourier-Mukai transforms of Lemmas \ref{lemma-fourier-mukai-1} and \ref{lemma-fourier-mukai-2} as well as their right adjoints preserve $\D$. 
		
		Firstly consider $\D = \Perf$. Note that all the relevant functors are compositions of pullbacks, upper-shrieks and pushforwards. Pullbacks always preserve perfect complexes, and our upper-shrieks all decompose as a pullback followed by tensoring with a perfect, by \ref{condition-2} and \ref{condition-4} of Definition \ref{def-fourier-mukai}, and so they also preserve perfects. It thus suffices to show that the pushforwards $Rp_*$, $j_*$ and $R\pi_*$ preserve $\Perf$. Now the question is local, so we may assume $Z_1$ is affine and that $p: E \rightarrow Z_1$ is a weighted projective stack. Lemma \ref{lemma-koszul-beilinson} implies that $\Dqc(E)$ is compactly generated by the complex $\O_E(1-\sum d_i)\oplus....\oplus \O_E$, and we have seen in Lemma \ref{lemma-cohomology-line-bundles} that its pushforward is perfect. \cite[Theorem 5.1]{neeman_duality} now implies that $Rp_*$ preserves $\Perf$. Now $j$ is the inclusion of an effective Cartier divisor, so $j_*$ preserves perfect complexes by \cite[Theorem 4.14 (4)]{perfect_stax} and \cite[Theorem 5.1]{neeman_duality}. Finally, to show $R\pi_*$ preserves $\Perf$, we may assume $X = \Spec R$ is affine and that the centre is strictly Koszul-regular, say $$\cal{A} = (f_0, d_0)+...+(f_n,d_n).$$ Lemma \ref{lemma-koszul-beilinson} tells us that $\Dqc(\X)$ is compactly generated by $\O_{\X}(1-\sum d_i)\oplus....\oplus \O_{\X}$, and we have seen in Proposition \ref{prop-sga6} that $R\Gamma(\X, \O_{\X}) = \O_{X}$. Moreover, by Corollary \ref{cor: coho-calc-Rees}, we have \[R\Gamma(\tilde{X}, \O(r)) = \cal{A}_r \oplus
            \frac{1}{x_0...x_n}R[t^{-1},x_0^{- 1},...,x_n^{- 1}] /(f_i-x_it^{-d_i})_r [-n]. \] for any $ r\in \Z$, and in particular for $1-\sum d_i\leq r  <0 $, it follows that $H^0(\X, \O_{\X}(r)) = R$, and that $H^n(\X, \O_{\X}(r)) = 0$. In particular $R\Gamma(\X, \O_{\X}(r))$ is perfect for any $1-\sum d_i\leq r  \leq0 $. \cite[Theorem 5.1]{neeman_duality} now implies $R\pi_*$ preserves $\Perf$.
		
		Now suppose $X$ is locally noetherian and $\D = \DCoh^-$ i.e. the category of pseudo-coherent complexes. As before, it suffices to show that the pushforwards $Rp_*$, $j_*$ and $R\pi_*$ preserve $\DCoh^-$. But all these morphisms are proper and concentrated, so the result follows from \cite[Theorem 1.2]{proper_coverings} and the $E_2$-hypercohomology spectral sequence. 
		
		Finally, if $X$ is locally noetherian and $\D = \DbCoh$, then all pullbacks preserve $\D$ since $p,j$ and $\pi$ are all perfect. Now we have established $Rp_*$, $j_*$ and $R\pi_*$ preserve $\DCoh^-$ and so it is obvious they preserve $\DCoh^+$, and hence they preserve $\DbCoh$. 
		
		We now make the elementary observation that in all cases, \[j_*p^*\Dqc(Z_1)\otimes\O(r) \cap \D(\X) = j_*p^*\D(Z_1)\otimes\O(r),  \] and \[L\pi^*\Dqc(X) \cap \D(\X) = L\pi^* \D(X). \] Indeed, this follows from the fact that all the Fourier-Mukai transforms of Lemmas \ref{lemma-fourier-mukai-1} and \ref{lemma-fourier-mukai-2} are fully faithful (so their unit maps, as left adjoints, are isomorphisms), and that they preserve $\D$.
		
		The result now follows from another elementary observation: Let $\sf{T}' \subseteq \sf{T}$ be a (triangulated) inclusion of triangulated subcategories and let $\sf{T} = \langle \sf{T}_1, \sf{T}_2 \rangle$ be a semi-orthogonal decomposition of $\sf{T}$. Let $R_i: \sf{T}\rightarrow \sf{T}_i$ denote the right adjoint of the inclusion of $\sf{T}_i$. If $R_i|_{\sf{T}'}: \sf{T}' \rightarrow \sf{T}_i$ factors through $\sf{T}_i \cap \sf{T}'$, then  $\sf{T}' = \langle \sf{T}' \cap \sf{T}_1, \sf{T}'\cap \sf{T}_2\rangle$ is a semi-orthogonal decomposition of $\sf{T}'$.
	\end{proof}
	Let us now turn our attention to Theorem \ref{thm: shriek}. 
    \begin{proof}[Proof of Theorem \ref{thm: shriek}]
        By \cite[Theorem A]{perfect_stax}, it follows that $\Dqc(X)$ is compactly generated. By Lemma \ref{lemma-lb-are-generators} so is $\Dqc(\tilde{X})$. Moreover, $R\pi_*$ preserves compact objects - since $X$ is concentrated this follows from Corollary \ref{cor: coho-calc-Rees}. Thus by Brown representability $\pi^!$ exists, and by \cite[Theorem 5.4]{neeman_duality} (whose proof, which uses only compact generation, applies here), we know $\pi^!$ is of the form \[\pi^!(-) = L\pi^*(-) \ltimes \omega_\pi, \] where $\omega_\pi := \pi^!\Ox$. Moreover, since $X$ has quasi-finite and separated diagonal so does $\tilde{X}$. Indeed, it suffices to check the relative diagonal has these properties, which can be done locally - quasi-finiteness follows from \cite[Lemma 1.1.2]{QuekRydh}, and separation is obvious. Thus by \cite[Lemma 0.1]{neeman_giant_ass_duality}, applied to a smooth cover of $X$, and Proposition \ref{prop: shriek-affine}, it follows $\omega_\pi$ is a line bundle.

        Now let us consider the counit map $\O_{\tilde{X}}\rightarrow \omega_{\pi}$ of the adjunction. This map is injective and the cokernel is supported exactly on $E$ - this can be verified locally, where it is obvious. Thus $\omega_\pi$ is of the form $\O(rE)$. That $r = -1 + \sum d_i$ can be verified locally also.
    \end{proof}
        
	\bibliographystyle{abbrv}
	\bibliography{weighted_orlov}
\end{document}